 \documentclass[12pt]{amsart}

\usepackage{amssymb, amsmath, amsfonts, amscd}
\usepackage[raiselinks=false,colorlinks=true,citecolor=blue,urlcolor=blue,
linkcolor=blue,bookmarksopen=true,pdftex]{hyperref}

\usepackage[dvipsnames]{xcolor}
\usepackage{enumerate}
\usepackage[mathscr]{eucal}
\newtheorem{theorem}{Theorem}[section]
\newtheorem{proposition}[theorem]{Proposition}
\newtheorem{lemma}[theorem]{Lemma}
\newtheorem{remark}[theorem]{Remark}

\renewcommand{\hom}{\textrm{Hom}}

\newcommand{\fix}{\textrm{Fix}}
\begin{document}
\baselineskip=15.5pt
\title[Cohomology of generalized Dold spaces]{Cohomology of generalized Dold spaces} 
\author[M. Mandal]{Manas Mandal} 

\address{The Institute of Mathematical Sciences, (HBNI), Chennai 600113.} 

\email{manasm@imsc.res.in}

\author[P. Sankaran]{Parameswaran Sankaran
}
\address{Chennai Mathematical Institute, Kelambakkam 603103.}
\email{sankaran@cmi.ac.in}
\subjclass[2010]{57R25}
\keywords{Generalized Dold manifolds,  complex Grassmann manifolds, cohomology algebra, complex $K$-theory}
\thispagestyle{empty}
\date{}

\begin{abstract}

Let $(X,J) $ be an almost complex manifold with a (smooth) involution $\sigma:X\to X$ such that $\fix(\sigma)\ne \emptyset$. Assume that 
$\sigma$ is a complex conjugation, i.e, the differential of $\sigma$ anti-commutes with $J$.  
The space $P(m,X):=\mathbb{S}^m\times X/\!\sim$ where $(v,x)\sim (-v,\sigma(x))$ was referred to as 
a generalized Dold manifold.
The above definition admits an obvoius generalization to a much wider class of spaces where $X, S$ are arbitrary 
topological spaces. The resulting space $P(S,X)$ will be called a 
generalized Dold space.  When $S$ and $X$ are CW complexes satisfying certain natural requirements, we obtain a CW-structure on $P(S,X)$.  Under certain further hypotheses, we determine the mod $2$ cohomology groups of $P(S,X)$.  
We determine the $\mathbb Z_2$-cohomology algebra when $X$ is (i) a torus manifold whose torus orbit 
space is a homology polytope,  (ii) a complex flag manifold.  One of the main tools is the Stiefel-Whitney 
class formula for vector bundles over $P(S,X)$ associated to  $\sigma$-conjugate complex bundles over $X$ when 
the $S$ is a paracompact Hausdorff topological space, extending the validity of the formula,  obtained earlier by Nath and Sankaran,
in the case of 
generalized Dold {\it manifolds}.   
\end{abstract}

\maketitle
\begin{center}
\end{center}

\section{Introduction} \label{intro}

The classical Dold manifold $P(m,n)$ is defined as the orbit space of the $\mathbb Z/2\mathbb Z$-action on $\mathbb S^m\times \mathbb CP^n$ generated by the involution $(v,[z])\mapsto (-v,[\bar{z}]), v\in\mathbb S^m,[z]\in\mathbb CP^n$.
Here $[\bar z]$ denotes $[\bar z_0:\cdots :\bar z_n]$ when $[z]=[z_0:\cdots:z_n]\in \mathbb CP^n.$                
 See \cite{dold}. 

Let $\sigma:X\to X$ be a complex conjugation on an almost complex manifold $(X,J)$, that is, $\sigma$ is an involution with non-empty fixed point set such that, for any $x\in X,$ the differential 
$T_x\sigma:T_xX\to T_{\sigma(x)}X$ satisfies the equation $J_{\sigma(x)}\circ T_x\sigma=-T_{x}\sigma \circ J_x$. See \cite[\S 24]{cf}.
The generalized Dold manifold $P(m,X)$ was introduced in \cite{ns} as the quotient of $\mathbb S^m\times X$ 
under the identification $(v,x)\sim (-v,\sigma(x))$.   The main focus in \cite{ns} was the study of manifold-properties 
of $P(m,X)$ such as the description of its tangent bundle, a formula for its total Stiefel-Whitney class, the 
(stable) parallelizability and related properties, and its cobordism class.  

Here, our aim is to compute the mod $2$-cohomology algebra of the generalized Dold manifolds.  While 
studying the homotopical/homological properites, it is natural to do away with stringent requirements 
such as $X$ to be an almost complex manifold.  
Also we replace $(\mathbb S^m,\textrm{antipode})$, by a pair $(S,\alpha)$ where 
$S$ is, say, a paracompact Hausdorff topological space and $\alpha:S\to S$ a fixed point free involution.  Likewise, $X$ is any Hausdorff topological space 
with an involution $\sigma:X\to X$ having a non-empty fixed point set.  Then $P(S,\alpha,X,\sigma)$  (or more briefly $P(S,X)$) is the space $S\times_{\mathbb Z_2} X=S\times X/\!\sim$ where $(v,x)\sim(\alpha(v),\sigma(x))$ which is the quotient of $S\times X$
by the free $\mathbb Z_2$-action generated by $\alpha\times \sigma$.  
When $S,X$ are CW complexes satisfying certain additional hypotheses, we show that, as $\mathbb Z_2$-vector spaces,  
$H^*(P(S,X);\mathbb Z_2)\cong H^*(Y;\mathbb Z_2)\otimes H^*(X;\mathbb Z_2)$ where $Y=S/\mathbb Z_2$.  
We also obtain the same result when $H^*(X;\mathbb Z_2)$ is generated by the Stiefel-Whitney classes 
of finitely many $\sigma$-conjugate (complex) vector bundles.  
See Proposition \ref{leray-hirsch}.  The notion of a $\sigma$-conjugate vector bundle 
will be defined in \S2.

We obtain a formula for Stiefel-Whitney classes of the real vector bundle over $P(S,X)$ associated to a $\sigma$-conjugate vector bundle over $X$  when $H^1(X;\mathbb Z_2)=0$.    This formula is applied to obtain the {\it ring} structure of 
$H^*(P(S,X);\mathbb Z_2)$ when $X$ is a torus manifold whose torus quotient is a homology polytope, or is a complex flag manifold. 
See Theorem \ref{torus} and Theorem \ref{gen-flagmanifolds}.  Note that the class of such torus manifolds contains the class of all quasi-toric manifolds as well as that of all smooth  complete toric manifolds.  Finally, as an application, we compute the equivariant cohomology ring 
$H^*_{\mathbb Z_2}(X;\mathbb Z_2)$ (see Theorem \ref{equivariant-cohomology}). 

%%%%%%%%%%%%%%%%%%%%%%%%%%%%%%%

Recently, Sarkar and Zvengrowski \cite{sarkar-zvengrowski} have constructed smooth manifolds which are a simultaneous generalization of projective product spaces 
due to Davis \cite{D} and of Dold manifolds and call them {\it generalized projective product spaces}.  These are $\mathbb Z_2$-quotients $P(M,N)$ 
of products $M\times N$ by the diagonal action where $M,N$ are manifolds admitting $\mathbb Z_2$-actions where the action on $M$ is assumed to be free.    When the fixed point set for $\mathbb Z_2$-action on $N$ is non-empty, they are also 
generalized Dold spaces in our sense.   
Sarkar and Zvengrowski also obtain results on $\mathbb Z_2$-cohomology in many special cases, including when $M$ is a product of spheres and $N$ is a quasi-toric manifold, besides results on manifold-properties of generalized product spaces.

%%%%%%%%%%%%%%%%%%%%%%%%%%%%%
\section{Generalized Dold spaces} 
%%%%%%%%%%%%%%%%%%%%%%%%%%%
Let $\alpha:S\to S$ be a fixed point free involution and let $\sigma :X\to X$ be involution with non-empty fixed point set.  We assume that $S,X$ are connected, locally path connected, paracompact, Hausdorff topological spaces.  We denote by 
$Y$ the quotient space $S/\!\!\sim$ where $v\sim \alpha(v), v\in S$.  
The space $P(S, \alpha,X,\sigma)$ (or more briefly $P(S,X)$ when there is no risk of confusion) is defined to be the quotient space 
$S\times X/\!\!\sim$ where $(v,x)\sim(\alpha(v),\sigma(x))$  and is called the {\it generalized Dold space}.   
It is a (smooth) manifold if $S,X$ (and $\alpha,\sigma)$ are smooth.  The quotient maps $q:S\to Y$ and $p:S\times X\to P(S,X)$ 
are covering projections with deck transformation groups isomorphic to $\mathbb Z_2$ generated by $\alpha$ 
and $\alpha\times \sigma$ respectively.   We denote by  $[v]\in Y$ the element $q(v)=\{v,\alpha(v)\}$ and similarly 
$[v,x]=p(v,x)=\{(v,x),(\alpha(v), \sigma(x))\}\in P(S,X).$  The first projection $S\times X\to Y$ induces a 
map $\pi:P(S,X)\to Y$ which is the projection of a locally trivial bundle with fibre space $X$.   Denote by $X^\sigma\subset X$ the fixed points of $\sigma$.   For each $x\in X^\sigma:=\textrm{Fix}(\sigma)$, we have a cross-section 
$s_x=s: Y\to P(S,X)$ where $s([v])=[v,x]$.  In fact, we have a well-defined 
map $i:Y\times X^\sigma\to P(S,X)$ defined as $i([v],x)=[v,x]$.  We observe that $i$ is an embedding.    

The above construction is functorial in $(S,\alpha)$ and in $(X,\sigma)$ for  maps $(S',\alpha')\to (S,\alpha)$ and $(X',\sigma')\to (X,\sigma)$ that are $\mathbb Z_2$-equivariant maps.

\subsection{$\sigma$-conjugate vector bundles}

Let $\sigma:X\to X$ be an involution on a path connected paracompact Hausdorff topological space.  
We assume that $X^{\sigma}$  is a non-empty proper subset of $X$.  
Let $\omega$ be a complex vector bundle over $X$.   A $\sigma$-{\it conjugation} on $\omega$ is an involutive 
bundle map $\hat\sigma: E(\omega)\to E(\omega)$ on the total space of $\omega$ that covers 
$\sigma$ and is conjugate  complex 
linear on the fibres of the projection $\pi_\omega: E(\omega)\to X$.   If such a bundle involution exists, we call $(\omega,\hat\sigma) $ (or more briefly $\omega$) a $\sigma$-conjugate bundle.

Let $(\omega,\hat\sigma)$ be a $\sigma$-conjugate vector bundle over $X$ with bundle projection $E(\omega)\to X$.  The zero-cross section $X\to E(\omega)$ is $\sigma$-equivariant.  In particular, $\fix(\hat\sigma)$ 
is non-empty.  
We obtain a {\it real} vector bundle 
$P(S, \omega)$,  more briefly denoted $\hat\omega$ when there is no danger of confusion, over $P(S,X)$ with projection $P(S,E(\omega),\hat{\sigma})\to P(S,X)$ defined as $[v,e]\mapsto [v,\pi_\omega(e)]$.  

Note that $\hat{\sigma}$ is also a $\sigma$-conjugation on the conjugate complex vector bundle $\bar\omega$ and we have an isomorphism of real vector bundles $\hat \omega\cong \hat{\bar \omega}$.

The above construction of  $\hat \omega$ over $P(S,X)$, as prolongation of a $\sigma$-conjugate complex vector bundle over $X$, can be extended to real vector bundles as follows.  Let 
$\eta$ be any {\it real} vector bundle over $X$ with a bundle involution $\hat\sigma$ that covers $\sigma$.
We denote by $\hat \eta$ the vector bundle over $P(S,X)$ with total space $P(S,E(\eta))=P(S,E(\eta),\hat\sigma)$ and 
bundle projection $\pi_{\hat\eta}:P(S,E(\eta))\to P(S,X)$ defined as $[v,e]\to [v,\pi_\eta(e)]~\forall v\in S, e\in E(\eta)$.  

\subsection{Functoriality}
Suppose that $\alpha_1: S_1\to S_1$ is a fixed point free involution and that $f: S_1\to S$ is $\mathbb Z_2$-equivariant, i.e., $f(\alpha_1(v_1))=\alpha f(v_1),v_1\in S$.   By the functoriality, we obtain a map $F: P(S_1, X)\to P(S,X)$. 
Explicitly $F$ is induced by the $\mathbb Z_2$-equivariant map 
$f\times id: S_1\times X\to S\times X$.  Also, we have a morphism of real vector bundles 
$P(S_1,\omega)\to P(S,\omega)$ where the map between the total spaces $\hat F: P(S_1,E(\omega), \hat\sigma)\to P(S,E(\omega),\hat\sigma)$  is again got by functoriality.  Explicitly, $\hat F([v_1,e])=[f(v_1),e]~\forall v_1\in S_1, e\in E(\omega)$.  We have the following commuting diagram:
\[\begin{array}{ccc}
P(S_1,E(\omega))& \stackrel{\hat F}{\to} & P(S, E(\omega))\\
\downarrow && \downarrow\\
P(S_1,X) &\stackrel{F}{\to}& P(S,X)\\
\end{array}
\]
 It follows that \[F^*(P(S,\omega))=P(S_1, \omega).\eqno(1) \]  

We have the following isomorphism of real vector bundle for any $\sigma$-conjugate complex vector bundle $(\omega,\hat \sigma)$ over $X$: 
\[ P(S,\omega)\cong \xi_\alpha \otimes P(S,\omega).\eqno(2)\]

An explicit bundle isomorphism can be obtained as follows:  Write $\hat \omega:= P(S,\omega)$.  The total space $E(\xi_\alpha \otimes \hat \omega)$ of 
$\xi_\alpha\otimes \hat\omega$ is a subspace of the quotient of $S\times X\times E(\epsilon_\mathbb R\otimes \omega)$ under the identification 
$(v, x, t\otimes  w)\sim (\alpha (v),\sigma(x),-t\otimes \hat\sigma(w))$.  
We write the 
equivalence class of $(v,x,t\otimes w)$ as $[v,x,t\otimes w]$. 
The space $E(\xi_\alpha\otimes \omega)$ 
consists of triples $[v,x,t\otimes w]$ where $\pi_{\epsilon_\mathbb R\otimes \omega}(t\otimes w)=x\in X.$
   A bundle map $f:E(\xi_\alpha\otimes \hat\omega)\to E(P(S,\omega))$ that covers the identity is obtained as $[v,x,t\otimes w]
\mapsto [v,x,\sqrt{-1} tw]$.   The routine verification is left to the reader.  (Cf. \cite[Proposition 1.4(iii) ]{ucci},  \cite[Lemma 2.7]{ns}.)

\subsection{The Stiefel-Whitney classes of $\hat \omega$} 
 Assuming that $H^1(X;\mathbb Z_2)=0$,  a formula for Stiefel-Whitney classes of $\hat \omega$ was 
obtained in \cite{ns} in the case of generalized Dold manifolds $P(\mathbb S^m,X,\sigma)$.  The proof used functorial properties 
of the Dold construction and the splitting principle for complex vector bundles over $X$.  These properties 
hold for any space $X$ as the manifold properties of $X$ play no role in the proof of the Stiefel-Whitney class 
formula.  Specifically, the following formula holds for $\sigma$-conjugate vetor bundles over $P(\mathbb S^n,X)$
for any CW complex $X$ with $H^1(X;\mathbb Z_2)=0$.  

There exist suitable cohomology classes $\tilde c_j(\omega)\in H^{2j}(P(\mathbb S^m,X);\mathbb Z_2)$ 
which restricts to the mod $2$ Chern class $c_j(\omega)=w_{2j}(\omega)\in H^{2j}(X)$ along the fibres of 
$P(\mathbb S^m,X)\to \mathbb RP^m$.  One has the following formula 
for the $i$th Stiefel-Whitney class of $ P(\mathbb S^m,\omega)=\hat \omega$ 
in terms of the $\tilde c_j(\omega)$:
 (cf. \cite[Prop. 2.11]{ns}):
%%%%%%%
\[w_{i}(\hat \omega)=\sum_{0\le j\le r} {r-j\choose i-2j} x^{i-2j} \tilde c_{ j}(\omega)\eqno(3)\]
%%%%%%%%
where $x:=w_1(\xi_\alpha)\in H^1(P(\mathbb S^m,X);\mathbb Z_2)\cong  H^1(\mathbb RP^m;\mathbb Z_2)\cong \mathbb Z_2$ is non-zero and 
$r=\textrm{rank}_\mathbb C\omega$.  (It is understood that the binomial coefficient ${a\choose b}=0$ if 
$b>a$ or if $a<0$ or $b<0$.) See \cite{ns}. 

We have 
$\tilde c_1(\omega)=w_2(\hat \omega)+{r\choose 2}x^2$ and $w_1(\hat \omega)=rx$.  It can be seen 
easily, using (2) and induction, that $\tilde{c}_j(\omega)$ is expressible as a polynomial $\tilde c_j(\omega)=Q_j(x^2,w_2(\hat\omega),w_4(\hat \omega),\ldots, w_{2j}(\hat\omega))$ for $1\le j\le r$.  Therefore,
for $0\le j<r$,  $w_{2j+1}(\hat\omega)$ can be expressed as a polynomial in $x$ and even Stiefel-Whitney 
classes $w_{2i}(\hat\omega), 0\le i\le j$: 
\[w_{2j+1}(\hat{\omega})=xP_j(x^2,w_2(\hat \omega),w_4(\hat \omega) , \ldots, w_{2j}(\hat\omega))\eqno(4)\] for a suitable polynomial $P_j=P_j(x^2,w_2, w_4,\ldots, w_{2j})$ of total degree $2j$ where $\deg(w_i)=i$.  As noted already, $P_0=r\in \mathbb Z_2$.  
Equations (3) and (4) are still valid when $X$ is any connected CW complex as long as $H^1(X;\mathbb Z_2)=0$.

We need to extend Equation (4)  to the more general 
context of the bundle $\hat \omega:= P(S,\omega)$ over $P(S,X,\sigma)$ under the assumption that 
 $Y$ (and hence $S$) and  $X$ are CW complexes (although one can relax these restrictions 
even further).

We proceed as follows:  First, extend the validity of (3) to the  case when $S=\mathbb S^\infty$ with antipodal involution $-id$ so that $Y=\mathbb RP^\infty$.  Then, any double cover $S\to Y$ can be classified by a map $\bar f: Y\to \mathbb RP^\infty$ 
and so we have an $\mathbb Z_2$-equivariant map $f:(S,\alpha)\to (\mathbb S^\infty,-id)$.  Hence we obtain 
formula (3) with $x=w_1(\xi_\alpha)$ for the bundle $P(S,\omega)$, using the naturality of Stiefel-Whitney classes and the isomorphism (1).  So, it remains  only to consider the case $(S,\alpha)=(\mathbb S^\infty,-id)$.  

Note that $P(\mathbb S^\infty,X)=\cup_{m\ge 1} P(\mathbb S^m,X)$.  Moreover, the inclusion 
$ P(\mathbb S^m,X)\hookrightarrow P(\mathbb S^\infty ,X)$ is an $(m-1)$-equivalence (as can be seen using the homotopy exact sequence) .  Since the spaces involved are CW complexes $i$ induces an isomorphism in cohomology 
up to dimension $m-1$.  Given any $\sigma$-conjugate complex vector bundle $\omega$ over $X$ of rank $r$  we 
merely choose $m>2r$.  It follows that the formula (3) holds for $P(\mathbb S^\infty,\omega)$.

 Using (4), or directly, we see that $P(S, \omega)$ is orientable if and only if $rank_\mathbb C(\omega)$ is even.

\section{Cell structure and cohomology}  
Let $S, X$ to be locally finite CW complexes.   Under appropriate hypotheses on the cell-structures on $S$ and $X$, we obtain a nice CW structure on $P(S,X)$ which has the property that the mod $2$ cellular boundary map vanishes.  We shall see that there are many 
classes of smooth manifolds $S,X$ where there are such CW structures.  

Denote by $(C_*(A),\partial^A)$ the cellular chain complex of a 
CW complex $A$ over $\mathbb Z_2$.    Every cell $e$ has a unique orientation mod $2$ and so defines a 
basis element, again denoted by $e$, of $C_q(A;\mathbb Z_2)=H_q(A^{(q)},A^{(q-1)};\mathbb Z_2)$. 
When $A$ is clear from the context we write $\partial $ instead of 
$\partial ^A$. 
In fact, we shall denote by the same symbol $\partial$ for the differential in any 
chain complexes when there is no danger of confusion which one is meant.

Let $X$ be a connected CW complex.   
Assume that $\sigma: X\to X$ is an involution which stabilizes {\it each} cell of $X$.   In particular 
the zero cells are contained in $\fix(\sigma)$. 
Since $\sigma$ stabilizes each cell of $X$,  we have $\sigma_*(e)=\pm e=e$ for each cell $e$ of $X$ and so 
$\sigma_*:C_*(X;\mathbb Z_2)\to C_*(X;\mathbb Z_2)$ is the identity map.

Let $S$ be a $\mathbb Z_2$-equivariant CW complex where $\alpha: S\to S$ is the generator of the $\mathbb Z_2$-action on $S$. Since $\alpha$ is fixed point free involution, 
no open cell is mapped to itself under $\alpha$. The CW-structure on $(S,\alpha)$ yields a CW structure on $Y$ consisting of one cell $q(e)=q(\alpha(e))$ for each pair of cells $e,\alpha(e)$ of $S$. On the 
other hand, any CW structure on $Y$ lifts to an 
equivariant CW-structure on $S$.  

Also we obtain the product CW structure on $S\times X$ which is $\mathbb Z_2$-equivariant where the 
$\mathbb Z_2$-action 
is generated by $\tau:=\alpha\times\sigma$.   It consists of cells which are products $e\times d$ where $e,d$ are 
cells of $S$ and $X$ respectively.    The induced CW structure on $P(S,X)$ consists of one 
cell $(e,d)$ for each pair of cells $e\times d,\tau(e\times d)=\alpha (e)\times\sigma( d)=\alpha(e)\times d$.  

Suppose that the cell structure on $X$ is perfect mod $2$, that is,
the differential in the cellular chain complex of $X$ with $\mathbb Z_2$-coefficients vanishes 
in all dimensions.  
  Since $\partial^X=0$ in all dimensions, we have 
$H_*(X;\mathbb Z_2)\cong C_*(X;\mathbb Z_2)$.    
 Our assumptions are trivially valid when $X$ has  
no odd dimensional cells. In this case the cellular boundary map $\partial$ vanishes even with $\mathbb Z$-coefficients.  

 Fix $r\ge 0$.  We have $C_r(S\times X;\mathbb Z_2)
\cong \oplus_{p+q=r} C_p(S;\mathbb Z_2)\otimes C_q(X;\mathbb Z_2)$.   Since $\partial ^X=0$, we have  
$\partial (e\otimes d)=\partial(e) \otimes d.$  Morevover $\tau_*(e\otimes d)=\alpha_*(e)\otimes
\sigma_*(d)=\alpha(e)\otimes d$.  Recall that $p$ denotes the covering projection $S\times X\to P(S,X)$.  Since $p\circ \tau=p$ we have the following commuting diagram 
of chain complexes:
%%%%%%%%%%%%
\[
\begin{array}{rcl}
C_r(S\times  X;\mathbb Z_2)& \stackrel{\tau_*}{\longrightarrow} & C_r(S\times X,\mathbb Z_2)\\
p_*\searrow && \swarrow p_*\\
&C_r(P(S,X);\mathbb Z_2)&\\
\end{array}
\]
We denote by $(e,d)$ the cell $p(e\times d)=p(\alpha( e)\times d)$ in $ P(S,X)$.  Let $\dim e=p,\dim d=q$.  Suppose that $\partial ^S(e)=\sum_{j\in J(e)} e_j\in C_{p-1}(S;\mathbb Z_2)$ where $J(e)$  is a suitable 
 subset of the indexing set for cells of $S$.    
Then, $\partial(e\times d)=\sum e_j\otimes d$ and so 
\[\partial((e,d))=\sum_{j\in J(e)} p_*(e_j\otimes d)
=\sum_{j\in J(e)} (e_j,d)\in C_{r-1}(P(S,X);\mathbb Z_2).\eqno(5)\]
    Since $q_*(\alpha(e_j))=q_*(e_j)\in C_{p-1}(Y)$, we have $(e_j,d)=(\alpha(e_j),d)$ for any cell 
$d$ of $X$.  If $\alpha(J(e))=J(e)$, then both $(e_j,d)$ and $(\alpha(e_j),d)$ occur in the sum (5) and 
cancel each other out and we  conclude that $\partial((e,d))=0$.  
 (Note that $e_j, \alpha( e_j)$ are distinct cells of $S$.)   We are ready to prove the following proposition. 
 For the notion of cohomology extension of fibre we refer the reader to \cite[Theorem 9, \S7, Chapter 5]{spanier}.

\begin{proposition}    \label{perfectcells}  Suppose that $(S,\alpha)$ is any $\mathbb Z_2$-equivariant CW complex such that the induced cell structure on $Y$ is perfect mod $2$.  Suppose that $X$ is a CW complex  such that (i) each skeleton $X^{(k)},k\ge 1$ is finite, 
(ii)  the cell-structure  is perfect 
mod $2$, and, (iii) each cell is mapped to itself by $\sigma:X\to X$.
Then the CW structure on $P(S,X)$ 
induced by the product CW structure on $S\times X$ is perfect mod $2$.
In particular, one has an isomorphism 
 $H^*(P(S,X);\mathbb Z_2)\cong H^*(Y;\mathbb Z_2)
\otimes H^*(X;\mathbb Z_2)$ of $H^*(Y;\mathbb Z_2)$-modules.
\end{proposition}
%%%%%%%%%%%%%%%%
\begin{proof} Let $\partial^S(e)=\sum_{j\in J(e)}e_j$.  
Since the induced cell structure on $Y$ is perfect,  we have $0=\partial^Y(q_*(e))=q_*(\partial ^S(e))=q_*\sum_{j\in J(e)}e_j$.  Since $q_*(\partial^S(e))=0$
we see that $J(e) $ 
is stable by $\alpha$.
Hence, from Equation (5), for any cell $d$ of $X$, we have $\partial ((e,d))=0$.  Thus the induced cell-structure 
on $P(S,X)$ is perfect mod $2$. 

Assume that $X$ is a finite CW complex.  
We have  $H_k(P(S,X);\mathbb Z_2)\cong C_k(P(S,X);\mathbb Z_2)
\cong \oplus_{i+j=k}C_i(Y;\mathbb Z_2)\otimes C_j(X;\mathbb Z_2)\cong \oplus _{i+j= k}H_i(Y;\mathbb Z_2)\otimes H_j(X;\mathbb Z_2)$.    Consequently $H_*(P(S,X);\mathbb Z_2)\cong H_*(Y;\mathbb Z_2)\otimes H_*(X;\mathbb Z_2)$.    Also $H^*(P(S,X);\mathbb Z_2)\cong H^*(Y;\mathbb Z_2)\otimes H^*(X;\mathbb Z_2)$ is a free module over $H^*(Y;\mathbb Z_2)$ via $\pi^*:H^*(Y;\mathbb Z_2)\to H^*(P(S,X);\mathbb Z_2)$.

In the general case, the inclusion $X^{(k)} \hookrightarrow X$ induces an inclusion $P(S,X^{(k)})\hookrightarrow 
P(S,X)$ which covers the identity map of $Y$.   The proposition follows from the observation that, for any $n\ge 1$, the inclusion-induced homomoprhism 
$H^n(P(S,X);\mathbb Z_2)\to H^n(P(S,X^{(k)});\mathbb Z_2)$ is an isomorphism for all $k>n$. 
  \end{proof}

We shall obtain another situation where the $X$-bundle $\pi:P(S,X)\to Y$ admits a $\mathbb Z_2$-cohomology extension of fibre.  
As observed already, since $X^\sigma\ne \emptyset$,  we 
have a cross-section $Y\to P(S,X)$ of the $X$-bundle and so it follows that 
$\pi^*:H^*(Y;\mathbb Z_2)\to H^*(P(S,X);\mathbb Z_2)$ is a monomorphism. 
We have the following result. For the notion of cohomology extension of fibre we refer the reader to \cite[Theorem 9, \S7, Chapter 5]{spanier}.

\begin{proposition} \label{leray-hirsch}  
  
Let $(\omega_j,\hat \sigma_j),1\le j\le r$, be 
$\sigma$-conjugate complex vector bundles over $(X,\sigma)$ such that 
the cohomology algebra $H^*(X;\mathbb Z_2)$ is generated by the mod $2$ Chern classes 
$c_q(\omega_j)\in H^{2q}(X;\mathbb Z_2), 1\le q\le \textrm{rank}(\omega_j), 1\le j\le r$.   Assume that $X$ satisfies 
any of the following:  (i) $\dim_{\mathbb Z_2}H^*(X;\mathbb Z_2) <\infty$, (ii) $X$ is a CW complex with finite 
$k$-skeleton for each $k\ge 1$ where each cell is stable by $\sigma$. 
Then we have an isomorphism of $\mathbb Z_2$-vector spaces:
\[ H^r(P(S,X);\mathbb Z_2)\cong \bigoplus_{p+q=r}H^p(Y;\mathbb Z_2)\otimes H^q(X;\mathbb Z_2). \eqno(6) \]
In particular, 
$H^*(P(S,X);\mathbb Z_2)$ is a free $H^*(Y;\mathbb Z_2)$-module with basis $B$ where $B$ is a $\mathbb Z_2$-basis for $H^*(X;\mathbb Z_2)$. 
\end{proposition}
\begin{proof}  Let $e_0\in S$.  
Set $\hat \omega_j:=P(S,\omega_j), 1\le j\le r$.  Then $\iota^*(\hat \omega_j)\cong \omega_j$ where $\iota$ is the inclusion $\iota: X\to P(S,X)$ is defined as $x\mapsto [e_0,x]$.   
So $w_{2i}(\hat\omega_j)$ restricts to $w_{2i}(\omega_j)=c_i(\omega_j)
\in H^{2i}(X;\mathbb Z_2)$.  Since the $c_i(\omega_j), 1\le i\le \textrm{rank}(\omega_j), 1\le j\le r$, generate 
the cohomology algebra $H^*(X;\mathbb Z_2)$, it follows that 
the homomorphism $H^*(P(S,X);\mathbb Z_2)\to H^*(X;\mathbb Z_2)$ is surjective.  Thus, the $X$-bundle 
$(P(S,X),Y,\pi)$ admits $\mathbb Z_2$-cohomology extension of the fibre.  If $\dim_{\mathbb Z_2} H^*(X;\mathbb Z_2)$ is finite, such as when $X$ is a finite CW complex, by Leray-Hirsch theorem 
we conclude that $H^*(P(S,X);\mathbb Z_2)\cong H^*(Y;\mathbb Z_2)\otimes H^*(X;\mathbb Z_2)$ as $H^*(Y;\mathbb Z_2)$-modules.  This proves (i). 

 Suppose that $X$ is a CW complex such that for any $k\ge 1$ the $k$-skeleton $X^{(k)}$ of $X$ is a finite CW complex. Moreover, $\sigma$ restricts to $X^{(k)}$. It follows from part (i) of the proposition that $H^*(P(S,X^{(k)});\mathbb Z_2)\cong H^*(Y;\mathbb Z_2)\otimes H^*(X^{(k)};\mathbb Z_2)$, which establishes the isomorphism (6).
The last assertion follows readily from (6). \end{proof}

\subsection{Sphere bundles}
Let $X=\mathbb S^n,n\ge 1,$ and let $\sigma:\mathbb S^n\to \mathbb S^n$ be the reflection map that sends $e_{n+1}\to -e_{n+1}$ and pointwise 
fixes $\mathbb S^{n-1}\subset \mathbb R^n=\{e_{n+1}\}^\perp$.  
The cell structure on $\mathbb S^n$ with one $0$-cell  $d_0=\{e_1\}$ and one (closed)  $n$-cell $d_n=\mathbb S^n$ is stable by $\sigma$.  With $(S,\alpha)$ as in Proposition \ref{perfectcells}, 
we have $H^*(P(S,\mathbb S^n);\mathbb Z_2)\cong 
H^*(Y;\mathbb Z_2)\otimes H^*(\mathbb S^n;\mathbb Z_2)$ which is a free $A:=H^*(Y;\mathbb Z_2)$-module of rank $2$.   We shall determine $H^*(P(S,\mathbb S^n);\mathbb Z_2)$ as an $A$-algebra.  The two cases $n\ge 2$ 
and $n=1$ need to be treated separately.

\begin{proposition}    Let $n\ge 2$.  Suppose that $S$ satisfies the hypothesis of Proposition \ref{perfectcells}. 
Let $A=H^*(Y;\mathbb Z_2)$. Then  
$H^{*}(P(S,\mathbb S^n);\mathbb Z_2)\cong A[u]/\langle u^2\rangle$ as an $A$-algebra where $\deg u=n$.
\end{proposition}
\begin{proof} We first prove the claim when $S=\mathbb S^m$, then extend it to the case $S=\mathbb S^\infty$.  The general case will be shown to follow from the case $S=\mathbb S^\infty$.  
  
Let $S=\mathbb S^m$, $\alpha=-id$ so that $Y=\mathbb RP^m$.  The  Proposition is obvious when $m<n$ and so we assume that $m\ge n$.   We shall denote $P(\mathbb S^m,\mathbb S^n)$ by $P_{m,n}$.  One has an inclusion 
$P_{k,n}\subset P_{m,n} $ for $0\le k\le m$, where $k=0$ corresponds to the fibre inclusion $\mathbb S^n\hookrightarrow P_{m,n}$.  

 Let $y$ be the non-zero element in $H^1(P_{m,n};\mathbb Z_2)\cong H^1(\mathbb{R}P^m;\mathbb{Z}_2)\cong \mathbb Z_2$. Then $A=H^*(\mathbb{R}P^m;\mathbb{Z}_2)=\mathbb{Z}_2[y]/ \langle y^{m+1} \rangle$.
 Observe that $y$ is Poincar\'e dual to the submanifold $P_{m-1,n}\hookrightarrow P_{m,n}$.  Also, one may 
obtain $P_{k,n}$ as the intersection of $(m-k)$-copies of $P_{m-1,n}$ in general position and so $y^{m-k}$ is 
the Poincar\'e dual of $P_{k,n}$.  In particular the Poincar\'e dual of the fibre $\mathbb S^n\hookrightarrow P_{m,n}$ 
equals $y^m$. 

Any $x\in \mathbb S^{n-1}=\textrm{Fix}(\sigma)$ defines a cross-section $s:\mathbb RP^m\to P_{m,n}$ of the sphere bundle projection $P_{m,n}\to \mathbb RP^m$.  Any two such sections are isotopic since $\mathbb S^{n-1}$ is connected. 
(Here we use the hypothesis that $n\ge 2$.) {\it 
We set $u=u_n$ to be the Poincar\'e dual of the submanifold $s:\mathbb RP^m\to P_{m,n}$.}  Since the intersection 
$s(\mathbb RP^m)\cap \mathbb S^n$ is transverse and is exactly one point, we see that $y^mu\in H^{m+n}(P_{m,n};\mathbb Z_2)\cong \mathbb Z_2$ is non-zero.   Also, taking two distinct fixed points $x,x'\in \mathbb S^{n-1}$, we obtain cross-sections $s,s'$ where 
$s(\mathbb RP^m)\cap s'(\mathbb RP^m)=\emptyset$.  This shows that $u_n^2=0$, completing the proof of the proposition in this 
case.

Next, let $S=\mathbb S^\infty$.   Choose $m>2n$ so that the inclusion $j: P(\mathbb S^m,\mathbb S^n)\hookrightarrow P( \mathbb S^\infty,\mathbb S^n)$ is an $(m-1)$-equivalence.  Let 
$U_n\in H^n(P(\mathbb S^\infty,\mathbb S^n);\mathbb Z_2)\cong H^n(P(\mathbb S^m,\mathbb S^n);\mathbb Z_2)$ 
be the element that corresponds to $u_n$.  Since $u_n^2=0$ and since $j^*$ is an isomorphism in dimension $2n$, 
it follows that $U_n^2=0$.  Thus taking $u=U_n$, the claim holds for $S=\mathbb S^\infty$.  

In the general case of $(S,\alpha)$, we have a classifying map $\bar f:Y\to \mathbb RP^\infty$ which classifies the double covering $S\to Y$. Let $f:S\to \mathbb S^\infty$ be a lift of $\bar f$.  Then we obtain a morphism of $\mathbb S^n$-bundles 
$F:P(S,\mathbb S^n)\to P(\mathbb S^\infty,\mathbb S^n)$ that covers $\bar f$ by functoriality.  Let $\mathbf{u}_n=F^*(U_n)$.  Since $H^*(P(\mathbb S^\infty,\mathbb S^n);\mathbb Z_2)\to H^*(P(S,\mathbb S^n);\mathbb Z_2)$ 
is a ring homomorphism, we have $\mathbf {u}_n^2=0$. Setting $u:=\mathbf{u}_n$, the Claim follows. 
\end{proof}

We turn to the case $n=1$.  Note that when $m=1$, $P_{m,1}$ is the Klein bottle.  It can be seen that, as in the 
special case of the Klein bottle, $P_{m,1}$ is a connected sum $\mathbb RP^{m+1}\#\mathbb RP^{m+1}$.  To see this, we let $J^+, J^-\subset \mathbb S^1\subset \mathbb R^2$ to be the closed arcs with end points $e_2, -e_2$ where 
$e_1\in J^+, -e_1\in J^-$.  Note that $\sigma$ stabilizes $J^+$ and $J^-.$ Then $P(\mathbb S^m,J^+), P(\mathbb S^m,J^-)\subset P_{m,n}$ are twisted $I$-bundles over $\mathbb RP^m$ with common boundary $\mathbb S^m$.  The projections of the $I$-bundle restricted 
to $\mathbb S^m$ is the double cover $\mathbb S^m\to \mathbb RP^m$.    Since $\mathbb RP^{m+1}$ is got by attaching an $m$-cell via the double cover $\mathbb S^m\to \mathbb RP^m$, it follows that $P_{m,1}$ is the connected sum $\mathbb RP^{m+1}\#\mathbb RP^{m+1}$ .   Hence $H^*(P_{m,1})\cong \mathbb Z_2[a,b]/\langle a^{m+2}, b^{m+2}, ab\rangle.$
The projection of the circle bundle $\pi: P_{m,1}\to \mathbb RP^m$ induces $\pi^*:H^*(\mathbb RP^m;\mathbb Z_2)
\to H^*(P_{m,1};\mathbb Z_2)$, defined by $y\mapsto a+b$, where $y$ is the generator of $A=H^*(\mathbb{R}P^m;\mathbb{Z}_2)$. (Note that $y=a+b\in H^1(P_{m,n};\mathbb Z_2)$ is the {\it only} element that satisfies the conditions $y^{m+1}=0, y^m\ne 0$.)    As an $A$-algebra, $H^*(P_{m,1};\mathbb Z_2)$ is isomorphic to $A[a]/\langle a^2-ay\rangle $.  When $m=\infty$, $A$ is a polynomial algebra in $y$ and we have $H^*(P_{\infty,1};\mathbb Z_2)\cong A[a]/\langle a^2-ay\rangle$.  

When $(S,\alpha)$ is as in Proposition \ref{perfectcells}, $P(S,\mathbb  S^1)=Y^+\cup Y^-$ where $Y^\pm=
P(S, J^\pm)$ are total spaces of twisted $I$-bundles $\gamma^\pm $ over $Y$ where $Y^+\cap Y^-=P(S,\mathbb S^0)\cong S$.  
We have a continuous surjection $\eta: P(S,\mathbb S^1)\to P(S,J^+)/S=T(\gamma^+)$ where $T(\gamma^+)$ stands for the Thom space of $\gamma^+$.  
We let $a\in H^1(P(S,\mathbb S^1);\mathbb Z_2)$ be the image of the Thom class $a\in H^1(T(\gamma^+);\mathbb Z_2)\cong \mathbb Z_2 $ under the homomorphism induced by $\eta$.  
Arguing as in the proof of the above proposition,  we obtain the following.  

\begin{proposition}
Let $(S,\alpha)$ be as in Proposition \ref{perfectcells}.  Then $H^*(P(S,\mathbb S^1);\mathbb Z_2)=A[a]\langle 
a^2-ay\rangle$ where $\deg a=1=\deg y$.  \hfill $\Box$
\end{proposition}

\section{$\mathbb Z_2$-cohomology algebra of certain generalized Dold spaces}

In this section we determine the cohomology ring of $P(S, X)$ (i) $X$ is a torus manifold under certain mild 
assumptions, which are satisfied by quasi-toric manifolds, (ii) $X$ is a complex flag manifolds and $S$ is paracompact. 
 Basic facts concerning torus manifolds that are required for our purposes will be recalled in \S 4.1.

We begin with families of compact connected smooth manifolds with involutions $(X,\sigma)$ satisfying the 
hypotheses of Proposition \ref{leray-hirsch}.   The basic examples are: (a) complex Grassmann manifolds, or more generally, 
complex flag manifolds $U(n)/(U(n_1)\times \cdots\times U(n_r))$ where $\sum n_j=n$, (b) smooth complete toric varieties, \cite{fulton} and, as we shall show below, their topological analogues quasi-toric manifolds \cite{dj} and 
torus manifolds (satisfying certain additional hypotheses) \cite{mp}.   
As for examples of $(S,\alpha)$ satisfying the hypothesis of Proposition \ref{perfectcells}, we may 
take $S$ to be oriented (real) Grassmann manifolds $\tilde G_{n,k}=SO(n)/SO(k)\times SO(n-k)$ where $\alpha: \tilde{G}_{n,k}\to \tilde{G}_{n,k}$ is the involution which reverses the orientation on each oriented $k$-plane in $\tilde G_{n,k}$.   Note that when $k=1$, $\alpha$ is the antipodal map of the sphere $\tilde G_{n,1}=\mathbb S^{n-1}$.   The quotient space $\mathbb RG_{n,k}$ is the real Grassmann manifold which admits a perfect cell-decomposition over $\mathbb Z_2$.   See \cite[Chapter 6]{ms}. 

We begin with the following lemma which will be needed in the sequel.  

Let $X$ be an oriented compact smooth manifold on which the torus $T\cong (\mathbb S^1)^k$ acts smoothly and effectively.  
Suppose that $H\subset T$ is a circle subgroup such that $F:=X^{H}=\{x\in X\mid t.x=x~\forall t\in  H\}$ is 
an oriented connected submaifold of codimension $2$.   Then $T$ stabilizes $F$. 
Let $\nu$ be the normal 
bundle of $F\hookrightarrow X$.     We put a Riemannian metric on $X$ that is invariant under $T.$     Then  
$\nu$ is the orthogonal 
complement of the tangent bundle $\tau F$ in $\tau X|_F$.  Moreover,  $\nu$ is a $T$-equivariant 
bundle over $F$.  
Note that $H$ acts on $\nu$ as bundle automorphisms since $F$ is point-wise
fixed by $H$.  
Since $F$ and $X$ are oriented, so is $\nu$ and we may (and do) regard $\nu$ as a complex line bundle. 

\begin{lemma}   \label{linebundle}   With notations as above, 
 the complex line bundle $\nu$ is the restriction to $F$ of  a $T$-equivariant 
complex line bundle $\omega$ over $X$. 
Moreover, one has a $T$-equivariant 
cross-section $s: X\to E(\omega)$ which vanishes precisely on $F$.  We have \\
$c_1(\omega)=[F]\in H^2(X;\mathbb Z)$ where $[F]$ denotes the 
the cohomology class dual to $F\hookrightarrow X$.
\end{lemma}
\begin{proof}   
 Without the equivariance 
part, the existence of $\omega$ and $s$ were proved in \cite{s} and the equality 
$c_1(\omega)= [F]$ was deduced.    So we need only construct $\omega$ and $s$ as  $T$-equivariant objects.  
As in the discussion before the statement of the lemma, we put a Riemannian metric on $X$ that is invariant under 
$T$.

Denote by $\omega_0$ the pull-back of $\nu$ to $D(\nu)$ via $\pi$ where $\pi:D(\nu)\to F$ is 
the projection of the unit disk bundle associated to $\nu$. 
Then $\omega_0:=\pi^*(\nu)$ is a $T$-equivariant complex line bundle which admits an equivariant cross-section $s_0:D(\nu)\to E(\omega_0)$ defined as follows:   Recall that $E(\omega_0)$ is the fibre product 
$\{(v, w)\in D(\nu)\times E(\nu)  \mid \pi(v)=\pi(w)\}\subset D(\nu)\times E(\nu)$ and that $T$ acts on 
$E(\omega_0)$ diagonally: $t.(v,w)=(t.v,t.w)$.    We have 
$s_0(v):=(v, v)~\forall v\in D(\nu)$.   Note that 
$s_0$ vanishes precisely when $v=0$, i.e., on the zero-cross section $F\to D(\nu)$.    
Moreover, $s_0(t.v)=(t.v,t.v)=t.(v,v)=t.s_0(v)$ for all $v\in D(\nu), t\in T$, showing that $s_0$ is 
$T$-equivariant.
Let $S(\nu)=\partial D(\nu).$   We have an isomorphism of complex line bundles $\phi: E(\omega_0|_{S(\nu)})\to S(\nu)\times \mathbb C$ defined as $\phi(v, zv) 
=(v,z||v||)~\forall v\in S(\nu), z\in \mathbb C$.     The section $\phi\circ s_0$ corresponds to the constant function 
$1$ (i.e., $\phi\circ s_0(v)=(v,1)~\forall v\in S(\nu)$).   Since the Riemannian metric on $X$ is $T$-invariant, we have $\phi(t.(v,zv))=\phi(t.v, z t.v))=(t.v, z||t.v||)=(t.v,z||v||)
=t. (v,z)=t.\phi(v,zv)~\forall (v,w)\in E(\omega_0), t\in T$.  This shows that, with the trivial $T$-action on 
$\mathbb C$ understood, $\phi$ is $T$-equivariant.  

Let $N\subset X$ be an equivariant tubular neighbourhood of $F\subset X$ which is $T$-equivariantly diffeomorphic to $D(\nu)$. (See \cite[\S2, Chapter VI]{bredon}.)  Identifying $N$ with $D(\nu)$ via such a diffeomorphism, we obtain a $T$-equivariant 
complex line bundle, again denoted $\omega_0,$ on $N$ which restricts to $\nu$ on $F$, and a  $T$-equivariant cross-section $s_0:N\to E(\omega_0)$ which vanishes precisely on $F$.  Moreover, we have an isomorphism $\phi: E(\omega_0|_{\partial N}) \to \partial N\times \mathbb C$.   We glue $E(\omega_0)$ and $(X\setminus \textrm{int}(N))\times \mathbb C$, the total space of the 
trivial line bundle $\varepsilon _\mathbb C$,  along $S(\nu) \times \mathbb C$ using $\phi$ to obtain a $T$-equivariant complex line bundle 
$\omega$ on $X$.  
The section $s_0$ extends to a $T$-equivariant section $s:X\to E(\omega)$ such that $s|_{X\setminus\textrm{int}(N)}$ corresponds to the constant function $x\mapsto 1\in \mathbb C$.   In particular, $s$ vanishes precisely on $F$.   
\end{proof}

\subsection{Torus manifolds}
The of notion of torus manifolds is due to Hattori and Masuda  \cite{hm}, \cite{m} .  We shall use the definition 
as given in  Masuda and Panov \cite{mp}.  In fact, we consider 
a restricted class of torus manifolds (in the sense of Masuda-Panov), namely, those torus manifolds where the torus action is locally standard  and the  orbit space is a homology polytope.   This restricted class itself is a generalization of the notion of quasi-toric manifolds due to Davis and Januskiewicz \cite{dj}, where the orbit space is a simple convex polytope.
We begin by recalling the basic definitions.

A {\it torus manifold} is an even dimensional 
 smooth compact orientable manifold $X$ on which an $n$-dimensional torus $T\cong 
U(1)^n$ acts smoothly and effectively with non-empty fixed point set where $\dim X=2n$.  
One says that the $T$-action on $X$  is {\it locally standard} if $X$ is covered by open sets $\{U\}$ that are 
equivariantly diffeomorphic to an open subset of $\mathbb C^n$ with the standard $T\cong U(1)^n$-action.  This means that, for each $U$, there exists an automorphism $\psi:T\to T$ and an embedding $f:U\to \mathbb C^n$ such that 
$f(t.u)=\psi(t)f(u)~\forall u\in U, t\in T$.   The orbit space $Q:=X/T$ is then an $n$-dimensional  manifold with corners (i.e., is modelled on $\mathbb R_{\ge 0}^n$).   It turns out that set of $T$-fixed points of $X$ is finite and their images 
are the vertices of $Q$.  It is easy to see (using the irreducible characters of the isotropy representation) that 
each $T$-fixed point is a connected component of the intersection of exactly $n$ 
distinct $T$-stable submanifolds each having codimension $2$ in $X$.  There are only finitely many $T$-stable 
codimension $2$ submanifolds in $X$, say $X_i, 1\le i\le m,$ and each of these are fixed by certain circle subgroup $S_i$ of $T$.  
These are the {\it characteristic submanifolds} of $X$. 
Their images, $Q_i:=X_i/T$  in $Q$ are the {\it facets} of $Q$, which have 
codimension $1$ in $Q$.  The characteristic submanifolds of $X$ are all orientable and are again torus manifolds 
under the $T/S_i$-action with orbit space $Q_i$.  

A {\it preface} of $Q$ is a non-empty intersection of a facets of $Q$.  A {\it face} of $Q$ is a connected component of a preface.  We regard 
$Q$ itself as (the improper) face; all the other faces are {\it proper}.   $Q$ is said to be {\it face-acyclic} if all its faces (including $Q$ itself) are acyclic, i.e., have the integral homology of a point.  If $Q$ is face acyclic, then every face contains a vertex of $Q$. 
 If $Q$ is face-acyclic and if every preface is a face, then $Q$ is called a {\it homology polytope}.
For example, 
when $X$ is a quasi-toric manifold, then $Q=X/T$ is a simple convex polytope, which is evidently a homology polytope.   (A (compact) convex polytope of dimension $n$ is {\it simple} if exactly $n$ facets meet at each vertex.)

A characteristic submanifold  $X_i\hookrightarrow X$ determines a circle subgroup $S_i\subset T$; namely,  the subgroup of $T$ that fixes each point of $X_i$.   Choosing 
an isomorphism $f_i: \mathbb S^1\cong S_i$ amounts to choosing a primitive vector $v_i\in \hom(\mathbb S^1, T)\cong \mathbb Z^n$ whose image equals $S_i$. Note that $v_i$ is determined up to sign corresponding to two isomorphisms  $\mathbb S^1\to S_i$ namely $f_i , f_i\circ \iota$ where $\iota(z)=z^{-1}~\forall z\in \mathbb S^1$.  The sign is determined once an orientation on $S_i$ is fixed. 

We shall denote the group of $1$-parameter subgroups $\hom(\mathbb S^1,T)$ by $\mathbf{ N}$ and the group 
of characters $\hom(T,\mathbb S^1)$ by $\mathbf N^\vee\cong \mathbb Z^n$.   One has a natural pariing 
$\langle.,.\rangle: \mathbf N^\vee\times \mathbf N\to \mathbf Z$ defined by $u\circ v(z)=z^{\langle u,v\rangle}$. 

An {\it omni-orientation} of $X$ is a choice of an 
orienation on $X$ and on each characteristic submanifold $X_i, 1\le i\le m.$   The orientations on $X, X_i$ determines a unique orientation
on the normal space to $T_xX_i\subset T_xX$ for any $x\in X_i$ which in tern leads to an orientation on $S_i.$  This 
determines a unique {\it primitive} vector $v_i\in \mathbf{N}$ whose image is $S_i$.

Fix an omni-orientation on $X$ and assume that $Q$ is a homology polytope.
Denote by $\mathcal Q$ the set of all facets of $Q$.
We obtain map $\Lambda: \mathcal Q\to \mathbf{ N}$ defined as $\Lambda(Q_i)=v_i, 1\le i\le m$.  
Local standardness of the $T$ action implies that $ \Lambda(Q_{i_1}),\ldots, \Lambda(Q_{i_n})$ 
is a  $\mathbb Z$-basis of $ \mathbf{N}$ whenever 
$Q_{i_1},\ldots, Q_{i_n}\in \mathcal Q$ meet at a vertex of $Q$.  The map $\Lambda$ is called the {\it characteristic function} 
 of $X$.   The pair $(Q,\Lambda)$ determines $X$ up to equivariant homeomophism assuming the vanishing of $H^2(Q;\mathbb Z)$.  In fact, let $X(Q,\Lambda)$ denote 
the space $T\times Q/\!\!\sim$ where $(t,q)\sim (t',q')$ if and only if $ q=q'$ and $t^{-1}t' $ belongs to the subgroup of $T$ generated by the one parameter subgroups $\Lambda(Q_{i_1}),\ldots, \Lambda(Q_{i_k})$ where $q$ is in the interior of the face $Q_{i_1}\cap\cdots\cap Q_{i_k}$.  
Then $X(Q,\Lambda)$ is a 
a smooth manifold on which $T$ acts on the left with orbit space $Q$.   Further $Q$ is naturally embedded in $X( Q,\Lambda)$ via the map $q\mapsto [1,q]$.   We regard $Q$ as a subspace of $X$.  The quotient map $X(Q,\Lambda)\to Q$
is therefore a retraction.  
It was shown in \cite[Lemma 4.5]{mp}, assuming the vanishing of $H^2(Q;\mathbb Z)$, that $X(Q,\Lambda)$ is equivariantly 
homeomorphic to $X$.  We identify $X$ with $X(Q,\Lambda)$.  

Let $\sigma: X(Q,\Lambda)\to X(Q,\Lambda)$ be the involution $[t,q]\mapsto [t^{-1},q]$.   (It is readily verified that 
this definition is meaningful.)   Then $X^\sigma=\{[t,q]\mid  t^2\in S_q, ~\forall q\in Q\}$, which is the analogue of a {\it small 
cover} in the context of quasi-toric manifolds; see \cite{dj}.   Suppose that $X_i$ is a characteristic submanifold of $X$, fixed 
by a subgroup $S_i\cong \mathbb S^1$ of $T$.  Let $q_i$ be in the interior of $Q_i$ (i.e., $q_i\in Q_i$ but not in any proper face of $Q_i$).  Then $S_i$ is the isotropy subgroup at $q_i$, and $X_i$ is the closure of the $T$-orbit of $Q_i.$ 
Since $\sigma(Q_i)=Q_i$, it follows that $\sigma(X_i)=X_i$.   

Let $\tilde{T}:=T\rtimes \langle \sigma\rangle$ be the semi-direct product where $\sigma.t=t^{-1}. \sigma$ in $\tilde T$.
 Then $\tilde T$ acts on $X$.  We put a Riemannian metric on $X$ which is invariant with respect to $\tilde T$.  
In particular $\sigma$ preserves the Euclidean metric on the normal bundle $\nu_i$ over $X_i$.  Let $N_i$ denote a 
tubular neighbourhood of $X_i$ that is stably by $\tilde T$ and is $\tilde T$-equivariantly diffeomorphic to $D(\nu_i)$, 
the disk bundle of $\nu_i$.  We denote by $\theta_i:N_i\to D(\nu_i)$ such a diffeomorphism. 
We have that $T\sigma$ restricts to a bundle map of 
the unit disk bundle $\pi_i:D(\nu_i)\to X_i$ covering $\sigma|_{X_i}$ and $\sigma(N_i)=N_i$.  As in the proof of Lemma \ref{linebundle}, 
we denote by $\omega_{i,0}$ the bundle $\theta_i^*( \pi_i^*(\nu_i))$ over $N_i$.   
We shall abuse notation and write $\pi_i$ for the projection $N_i\to X_i$.  This yields an involution $\hat\sigma_{i,0}$ on $\omega_{i,0}$ defined as 
$(u, v)\mapsto (\sigma(u), T_{\pi_i( u)}\sigma (v))$ where $u\in N_i, $ and $v\in D(\nu_i)$ is in the fibre of $D(\nu_i)\to X_i$ over $\pi_i( u)$.

\begin{lemma}  \label{tsigma}
(i) The bundle map $T\sigma|_{D(\nu_i)}$ is orientation reversing on $\nu_i$.\\
(ii) The involution $\hat\sigma_{i,0}:E(\omega_{i,0})\to E(\omega_{i,0})$ is a complex conjugation that 
covers $\sigma|_{N_i}$.
\end{lemma}
\begin{proof} (i) It suffices to show that $\sigma$ is orientation 
reversing on the fibre of $N_i\cong D(\nu_i)\to X_i$ over a point $q_i\in Q_i$.  Since any neigbouhood of 
$Q_i$ meets the interior of $Q$, we see that $N_i\cap int(Q)\ne \emptyset$.  Let $q\in N_i\cap int(Q)$.  
Since $\sigma(q)=q$, and since $\sigma$ is fibre preserving on $N_i\to X_i$, it follows that $q$ is in the fibre $D_{q_i}$ over 
a $q_i\in Q_i$.   Let $t_0\in S_i$ be the unique order two element. Then $[t_0, q]\in D_{q_i}$ and $\sigma([t_0,q])=[t_0,q]\ne [1,q]=q$ and, moreover, no other point in the $H_i$-orbit of $q$ is fixed by $\sigma$.  This shows that $\sigma$ is the reflection of the disk $D_{q_i}$ about the `diameter' of the disk $D_{q_i}$ through $q$.  Hence $\sigma$ is orientation 
reversing.

\noindent 
(ii) Fixing orientation on $\nu_i$, we obtain a reduction of the structure group of $\nu_i$ to  $SO(2)\cong U(1)$, 
making $\nu_i$ a complex line bundle.  Since $T\sigma|_{E(\nu_i)}$ preserves the Euclidean metric on $\nu_i$, 
and is an involution, we conclude that it is a complex conjugation covering $\sigma|_{X_i}$.  The same argument 
applied to the pull-back of $\nu_i$ via the projection $\pi_i: N_i\to X_i$ of the disk bundle shows that $(p,v)\mapsto (\sigma(p), T_{\pi_i(p)}\sigma(v))$ is a complex conjugation of $\pi^*_i(\nu_i)=\omega_{i,0}$ covering  $\sigma|_{N_i}:N_i\to N_i$.  
\end{proof}

Taking $F=X_i\subset X$, a characteristic submanifold in Lemma \ref{linebundle}, we obtain a $T\rtimes \langle \sigma\rangle $-equivariant 
complex line bundle $\omega_i$ over $X$ that extends the normal bundle $\nu_{i}$ over $X_i$, and a $T$-equivariant cross-section $s_i:X\to E(\omega_i)$ which vanishes precisely on $X_i$.  
In fact $\omega_i|_{N_i}=\omega_{i,0}$ and $\omega_{i}|_{X\setminus int(N_i)}=\varepsilon_\mathbb C$, the trivial 
complex line bundle.  

We claim that the complex conjugation $\hat{\sigma}_{i,0}$ on $\omega_{i,0}$  extends to a complex 
conjugation on $\omega_i$ that covers $\sigma:X\to X$.
Explicitly, we let 
$\hat\sigma_i$ to be the standard complex conjugation on the trivial bundle over $(X\setminus int(N_i))$.  It 
remains to show that, under the identification of $\omega_{i,0}|_{\partial N_i}$ with the trivial bundle  
via the cross-section $s_0:N_i\to E(\omega_i)$, the restriction 
$\hat{\sigma}_{i,0}|_{\partial N_i\times \mathbb C}$ is the 
 same as the standard complex conjugation on $E(\omega_{i,0}|_{\partial N_i})$.   This is clear since 
$s_0|_{\partial N_i}$ corresponds to the constant function $x\mapsto 1$ as noted in the proof of Lemma 
\ref{linebundle}.  
 Thus we have proved the following.

\begin{lemma}  \label{sigmaconjugatelinebundles} With the above notations,  
one has a $\sigma$-conjugation $\hat{\sigma}_i:E(\omega_i)\to E(\omega_i)$ for each $1\le i\le m$. 
Moreover, the Chern class $c_1(\omega_i)\in H^2(X;\mathbb Z)$ equals the cohomology class Poincar\'e 
dual to $X_i\hookrightarrow X$, i,e, $c_1(\omega_i)=[X_i]\in H^2(X;\mathbb Z)$.  
 \hfill $\Box$
\end{lemma}

In view of Lemma \ref{sigmaconjugatelinebundles} and the fact that the cohomology algebra $H^*(X;\mathbb Z_2)$ is generated by 
the classes $[X_i], 1\le i\le m$, the hypotheses of Proposition \ref{leray-hirsch}  are satisfied.  This leads to a 
description of the chomology $H^*(P(S,X);\mathbb Z_2)$ as a module over $H^*(Y;\mathbb Z_2)$ for any pair $(S,\alpha)$ 
where $S$ is a paracompact space and $Y=S/\mathbb Z_2$.
 Our aim is to describe $H^*(P(S,X);\mathbb Z_2)$ as an $H^*(Y;\mathbb Z_2)$-algebra in terms of generators and relations. 

Recall from \cite{mp} that the {\it integral} cohomology ring $H^*(X;\mathbb Z)$ is the quotient of the polynomial algebra  $\mathbb Z[x_1,\ldots, x_m]$ modulo the ideal $I$ generated by the following two types of elements:
\[
(i) ~x_{j_1}\ldots x_{j_r}=0~\textrm{~whenever~} Q_{j_1}\cap \cdots \cap Q_{j_r}=\emptyset,\]
\[ (ii)  ~\sum_{1\le j\le m} \langle u, v_j\rangle x_j=0~\forall u\in \hom(T,\mathbb S^1),\]
where $v_j=\Lambda(Q_j)\in \mathbf N$.
The element $x_j$ corresponds to $[X_j]\in H^2(X;\mathbb Z)$.  In particular, $X$ satisfies the hypothesis of Proposition \ref{leray-hirsch} and we have 
$H^*(P(S,X);\mathbb Z_2)\cong H^*(Y,\mathbb Z_2)\otimes H^*(X;\mathbb Z_2)$.  

For $u\in \mathbf N^\vee$, consider the complex line bundle $\omega_u:=\otimes_{1\le j\le m}\omega_j^{\otimes a_j}$ 
where $a_j=\langle u,v_j\rangle\in \mathbb Z$.   
Then $\omega_u$ is isomorphic to the trivial complex line bundle 
since $c_1(\omega_u)=\sum a_jc_1(\omega_j)=\sum a_j[X_j]=0$ in view of the relation (ii) above.  
Using the fact that $\bar \eta\cong \hom_\mathbb C(\eta,\epsilon_\mathbb C)$ and $ \eta\otimes \eta_1$  are $\sigma$-conjugate complex vector bundles when $\eta, \eta_1$  are $\sigma$-conjugate vector bundles 
 yields that $\omega_u$ is a $\sigma$-conjugate line bundle.  (See \cite[Example 2.2(iv)]{ns}.) In fact, in the case of $\bar\eta$, the $\sigma$-conjugation $\hat\sigma$ of $\eta$ is also a  $\sigma$-conjugation 
of $\bar\eta$. 

We have a $T\rtimes \mathbb Z_2$-equivariant cross-section $s_j:X\to E(\omega_j)$ whose zero locus equals $X_j$.  
We write $s_j^a$ to denote the corresponding cross-section of $\omega_j^{\otimes a}$  for $a\ge 1$; when 
$a<0$,  we set $s_j^a:=s_j^{ |a|},$ regarded as a section of $\bar\omega^{\otimes |a|}_j$.  (Note that $E(\omega)=E(\bar \omega)$). When $a=0$, $\omega_j^a$ is the trivial bundle and $s_j^a$ corresponds to the 
constant map $X\to \mathbb C$ to $x\mapsto 1$.

Let $\tilde x_j:=w_2(\hat\omega_j)\in P(S,X)$ where $\hat \omega_j:=P(S,\omega_j)$.   Then $\sum b_j\tilde x_j$ restricts 
to $\sum b_j x_j=w_2(\otimes \omega_j^{\otimes b_j})$ along the fibres of the $X$-bundle $P(S,X)\to Y$.  Let 
$p\in \fix(\sigma)$. 
 Denote by $s_{p}$ the cross section $Y\to P(S,X)$  defined as $[v]\mapsto [v,p]$.   
Since $H^1(X;\mathbb Z)=0$,  we have 
$H^2(P(S,X);\mathbb Z_2)
\cong H^2(Y;\mathbb Z_2)\oplus H^2(X;\mathbb Z_2)$ by Proposition \ref{leray-hirsch}.  

{\it Claim:}   $s_{p}^*(\tilde x_j)=0$ in $H^2(Y;\mathbb Z_2)$.  \\
To see this, note that $s_{p}:Y\to P(S,X)$  factors as follows: $Y\stackrel{\cong}{\to} P(S,p)\hookrightarrow P(S,X) .$  Now 
$\hat \omega_j|_{P(S,p)}=\xi_\alpha\oplus \epsilon_\mathbb{R}$ since $\omega_j|_{p}\cong \epsilon_\mathbb C=\{x_0\}\times \mathbb C$ (with standard complex conjugation). 
So $s_{p}^*(\tilde x_j)=w_2(\xi_\alpha\oplus \epsilon_\mathbb{R})=0$, as claimed.

It follows from the above Claim that $w_2(\otimes \hat \omega_j^{\otimes b_j})
=\sum b_j \tilde x_j.$   Taking $b_j=\langle u,v_j\rangle$ and using the isomorphism $P(S,\otimes \omega_j^{\otimes b_j})
\cong 
P(S, \omega_u)\cong P(S,\epsilon_\mathbb C)$ where the trivial complex line bundle has the standard conjugation, 
we obtain that, for any $u\in \mathbf N^\vee$,
\[\sum_{1\le j\le m} \langle u,v_j\rangle \tilde x_j=0.\eqno(7)\]

Next, suppose that $\cap_{1\le q\le r} Q_{j_q}=\emptyset$.   The Whitney sum $\omega:=\oplus_{1\le q\le r} \omega_{j_q}$ admits a cross-section $s:X\to E(\omega)$ given 
by $s(x)=(s_{j_1}(x),\ldots,s_{j_r}(x))$.  Clearly $s$  vanishes along $\cap_{1\le q\le r} X_{j_q}=\emptyset$, that is, $s$ is nowhere vanishing and so 
we obtain a splitting $\omega\cong \eta\oplus \epsilon_\mathbb C$.  The $\sigma$-conjugations on each summand $\omega_{j_q}$ 
of $\omega$ put together yields a $\sigma$-conjugation on $\omega$. 
Since $s$ is $T\rtimes\mathbb Z_2$-equivariant, 
the $\sigma$-cojugation on $\omega$ restricts to $\sigma$-conjugations on $\eta$ and $\epsilon_\mathbb C$, and, on the latter it is the standard conjugation. 
  It follows that $\hat \omega=\oplus_{1\le q\le r}\hat \omega_{j_q}=
\hat \eta\oplus \epsilon_\mathbb R\oplus \xi_\alpha.$ 
Hence the top Stiefel-Whitney class of $\hat \omega$ is zero. That is,
\[\prod_{1\le q\le r} \tilde x_{j_q}=0~\textrm{~whenever~$Q_{j_1}\cap \cdots\cap Q_{j_r}=\emptyset$~}.\eqno(8)\]

Let $A=H^*(Y;\mathbb Z_2)$ and let $A[\tilde x_1,\ldots, \tilde x_m]$ denote the polynomial algebra in the {\it indeterminates} $\tilde x_1,\ldots, \tilde x_m$.
As a consequence of (7) and (8) we obtain the following.

\begin{theorem}  \label{torus}
Let $X=X(Q,\Lambda)$ be a $T$-torus manifold where $X/T=Q$ is a homology polytope with $m$ facets.   Let $\sigma:X\to X$ be the involution $[t,q]\mapsto [t^{-1},q]$.   %Assume that $H^k(Y;\mathbb Z_2)$ is finite dimensional for each $k\ge 0$.  
Then, with the above notations,
 $H^*(P(S,X);\mathbb Z_2)$ is isomorphic, as an $A=H^*(Y;\mathbb Z_2)$-algebra, to the quotient 
$ R(Q,\Lambda):=A[\tilde x_1,\ldots, \tilde x_m]/I$ where 
the ideal $ I=I(Q,\Lambda)$ is generated by the following two types of elements: \\
(i) $\sum_{1\le j\le m}\langle u,v_j\rangle \tilde x_j,~u\in \mathbf N^\vee$, and, \\
(ii) $\prod_{1\le q\le r} \tilde x_{j_q}$ whenever $Q_{j_1}\cap \cdots \cap Q_{j_r}=\emptyset$. \\ The isomorphism is 
given by 
 $\tilde x_j\mapsto w_2(\hat \omega_j)\in H^2(P(S,X);\mathbb Z_2)$. 
\end{theorem}
\begin{proof} 
From the description of the cohomology ring of $X$ and the definition of $R(Q,\Lambda)$ it is clear that one has an 
isomorphism $R(Q,\Lambda)\cong H^*(X;A)\cong A\otimes_{\mathbb Z_2} H^*(X;\mathbb Z_2)$ of graded $A$-modules (by the universal coefficient theorem).   In particular $R(Q,\Lambda)$ is a free $A$-module of rank equal to 
$\dim _{\mathbb Z_2}H^*(X;\mathbb Z_2)<\infty$.    In fact, any $\mathbb Z_2$-basis consisting of monomials in $w_2(\omega_j)$ 
lifts an $A$-basis for $R(Q,\Lambda)$ got by replacing $w_2(\omega_j)$ by $\tilde x_j$. 

We have a well-defined $A$-algebra homomorphism $\theta: R(Q,\Lambda)\to H^*(P(S,X);\mathbb Z_2)$ defined by $\tilde x_j\mapsto w_2(\hat\omega_j), 1\le j\le m,$ in view of 
Equations (7) and (8). 
By Proposition \ref{leray-hirsch}, we see that $\theta$ is a surjective homomorphism of $A$-modules.   By the observation made above, 
as an 
$A$-module homomorphism, $\theta$ maps an $A$-basis to an $A$-basis and hence is an isomorphism.   
\end{proof}

\subsection{Grassmann manifolds and related spaces}

In this section our aim is to describe the $\mathbb Z_2$- cohomology ring of  
generalized Dold manifold $P(S, G_{n,k})$, which is fibred by the complex 
Grassmann manifold $X:=G_{n,k}=G_k(\mathbb C^n)$.  
The fixed point set of the (usual) complex conjugation $\sigma$ on $G_{n,k}$ is the real 
Grassmann manifold $\mathbb RG_{n,k}$.    Let $\xi=\xi_\alpha$ denote the line bundle associated to the double cover $S\to Y$.  We shall 
also denote by $\xi$ the line bundle over $P(S, G_{n,k})$ obtained as the pull back $\pi^*(\xi)$ via the projection $\pi:P(S,X)\to Y$ of the 
$G_{n,k}$-bundle.  Denote by $\gamma_{n,k}, \beta_{n,k}$ the tautological $k$-plane bundle and its orthogonal complement bundle, which is of 
rank $(n-k)$.  Note that the complex conjugation on $\mathbb C^n$ yields $\sigma$-conjugations on $\gamma_{n,k}$ and on 
$\beta_{n,k}$.  Moreover, we have an isomorphism 
\[\gamma_{n,k}\oplus \beta_{n,k}\cong \epsilon^n_\mathbb C.\] 
This yields an isomorphism of real vector bundles: (cf. \cite[Example 2.4(ii)]{ns})
\[\hat\gamma_{n,k}\oplus \hat\beta_{n,k}\cong n\xi_\alpha\oplus n\epsilon_\mathbb R\eqno(9)\]
where $\hat \omega:=P(S,\omega)$.   Consequently, the following relation among the Stiefel-Whitney classes holds, where 
$y=w_1(\xi_\alpha)$:
\[w(\hat \gamma_{n,k}).w(\hat \beta_{n,k})=(1+y)^n.\]
We rewrite the above relation in terms of Stiefel-Whitney polynomials:
\[w(\hat \beta_{n,k},t)=(1+yt)^n.w(\hat \gamma_{n,k},t)^{-1}=\sum_{j\ge 0} a_j t^j=a(t)\eqno(10)\]
where $a_j:=a_j(y, w_1(\hat\gamma_{n,k}), \ldots, w_{j}(\hat \gamma_{n,k}))$ is the homogeneous polynomial 
of (total) degree $j$ in 
$(1+y)^n.w(\hat \gamma_{n,k})^{-1}$. 

Since $H^1(P(S,G_{n,k});\mathbb Z_2)\cong H^1(Y;\mathbb Z_2)=\mathbb Z_2y$, it is clear that $w_1(\hat\gamma_{n,k}),w_1(\hat\beta_{n,k})\in \mathbb Z_2y$.   Moreover,  $w_1(\hat\gamma_{n,k})=0$ (resp. 
$w_1(\hat\beta_{n,k})=0$) if and only if $k$ is even (resp. $n-k$ is even) as a consequence of Equation (4) (or by a 
direct argument).  Hence 
we see that, $w_1(\hat\beta_{n,k})=ny+w_1(\hat\gamma_{n,k})\in \mathbb Z_2y$.    This 
also follows from the above equation.  
 Using Equation (4) and induction,  
we see that the Stiefel-Withney class $w_{j}(\hat \beta_{n,k})=a_j$ is expressible as a polynomial in 
$y,w_2(\hat \gamma_{n,k}),\ldots, w_{2i}(\hat\gamma_{n,k})$ for all $j$ where $i=\lfloor j/2\rfloor$.  
We note that, by degree considerations, $a_j$ is divisible by $y$ when $j$ is odd.  

We may view (10) as {\it defining} $w_{j}(\hat \beta_{n,k}), 1\le j\le 2(n-k),$ as the polynomial $a_{2j}$.  
Since $w_{j}(\hat \beta_{n,k})=0$ for $j>2n-2k$, Equation (10) leads to the relations 
\[a_{j}=a_j(y, w_2(\hat\gamma_{n,k}),\ldots,w_{2k}(\hat\gamma_{n,k}))=0, j>2n-2k.\eqno(11)\]   

Let $I\subset \mathbb Z_2[y, w_{2i}(\hat\gamma_{n,k}); 1\le i\le k]$ be the ideal generated by 
$a_j, j>2n-2k$.  Suppose that 
the height of $y$ equals $N\in \mathbb N$.  Then $(1+yt)^{-1}=\sum_{0\le j\le N}y^jt^j$ and we have 
$a(t). (\sum_{0\le j\le N} y^jt^j)w(\hat\gamma_{n,k},t)=1$.   It follows that $a_{N+2n+j}$ is in the 
ideal generated by $a_{2n-2k+i},1\le i\le N+2k,$ for all $j\ge 1$ and so $I$ is generated by 
$a_{2n-2k+i}, 1\le i\le N+2k$.  Moreover, it is easily seen that 
$a_{2n+i}$ is in the ideal generated by $y, a_{j}, 2n-2k+1\le j\le 2n$ for all $i\ge 1$.

Consider the graded polynomial algebra  $R:=\mathbb Z_2[y,\hat w_{2j};1\le j\le k]$ in the indeterminates $y, 
\hat w_{2j}, 1\le j\le k,$ where $\deg y=1, \deg \hat w_{2j}=2j$.   We regard 
$a_{2j}= a_{2j}(y, \hat w_{2}, \ldots,  
\hat w_{2j})$ as elements of $R$.   

\begin{lemma}\label{regularsequence}
The elements $y, a_{2j}\in R, n-k<j\le n,$ form a regular sequence in $R$ and $R/\langle y, a_{2j},n-k<j\le n\rangle 
\cong H^*(G_{n,k};\mathbb Z_2)$.  
\end{lemma} 
\begin{proof} 
To see this, it suffices 
to show that $\bar a_{2j}, n-k<j\le n$ is a regular sequence in $\bar R:=R/\langle y\rangle$ where $\bar a_{2j}:=a_{2j} \mod \langle y\rangle \in \bar R$.  Note that $\bar a_{2j}=h_{2j}(\hat w)$ where $h_{2j}(\hat w)$ denotes the complete symmetric polynomial of degree $2j$ in $\hat w_2,\ldots, w_{2k}$, that is, $h_{2j}$ equals the coefficient of 
$t^{2j}$ in $(1+\hat w_2t^2+\cdots+\hat w_{2k}t^{2k})^{-1}$.  One can show, as in \cite[\S23]{bott-tu}, that $\bar a_{2j}, n-k<j\le n$, is a 
regular sequence in $\bar R$ and that $\bar R/\langle a_{2j},n-k<j\le n\rangle\cong H^*(G_{n,k};\mathbb Z_2)$.
\end{proof}

 It follows from the above lemma that $y^{m+1}, a_{2j},n-k<j\le n$, is a regular sequence in $R$ and that the quotient $R/\langle y^{m+1}, a_{2j};n-k<j\le n\rangle$ is isomorphic to $\mathbb Z_2[y]/\langle y^{m+1}\rangle \otimes 
H^*(G_{n,k};\mathbb Z_2)=H^*(\mathbb RP^m;\mathbb Z_2)\otimes H^*(G_{n,k};\mathbb Z_2)$ as graded 
$\mathbb Z_2$-vector spaces.  We are ready to prove the following proposition.
Recall that $\xi_\alpha=\xi$ denotes the pull-back of the Hopf line bundle over $\mathbb RP^m$ via the projection of the $X$-bundle $P(\mathbb S^m,X)\to \mathbb RP^m$.

\begin{proposition}  \label{sphere-gnk}
With the above notations, 
the $\mathbb Z_2$-cohomology algebra $H^*(P(\mathbb S^m, G_{n,k});\mathbb Z_2) $ is isomorphic to $R/I$ 
where $I$ is the ideal generated by $y^{m+1}, a_{2j}(\hat w),n-k<j\le n,$ where $\hat w_{2j}$ corresponds to $w_{2j}(\hat \gamma_{n,k})$ and $y$ to $w_1(\xi)$. 
\end{proposition}
\begin{proof}
Consider the homomorphism $\eta: R\to H^*(P(\mathbb S^m,G_{n,k});\mathbb Z_2)$ of rings defined as 
$\eta(\hat w_{2j})=w_{2j}(\hat \gamma_{n,k})$ and $\eta(y)=w_1(\xi)$.  By Equation (10) and Proposition \ref{leray-hirsch}, $\eta$ is surjective.  It follows from  Equation (11) that $\eta$ factors as $ R\to R/I\stackrel{\bar \eta}{\to} H^*(P(\mathbb S^m;G_{n,k});\mathbb Z_2)$.  By the discussion preceding the statement of the proposition, we see that $\bar \eta$ is an isomorphism since 
$R/I$ and $H^*(P(\mathbb S^m,G_{n,k});\mathbb Z_2)=H^*(\mathbb RP^m;\mathbb Z_2)\otimes H^*(G_{n,k};\mathbb Z_2)$ have the same dimension.  
\end{proof}

Next we prove the following theorem. %follows from Theorem \ref{ringstructure}.

\begin{theorem}  \label{gen-grassmann}
We keep the above notations.  Suppose that $S$ is paracompact. \\
The cohomology algebra $H^*(P(S,G_{n,k});\mathbb Z_2)$ is isomorphic, as an $H^*(Y;\mathbb Z_2)$-algebra, to 
$H^*(Y;\mathbb Z_2)[\hat w_2,\ldots, \hat w_{2k}]/\mathcal I$, where 
$\mathcal I$ is generated by $a_{2j}, n-k<j\le n,$ under an isomorphism 
that maps $\hat w_{2j}$ to $w_{2j}(\hat \gamma_{n,k})$.  
\end{theorem}
\begin{proof}   Set $\mathcal R:=H^*(Y;\mathbb Z_2)[\hat w_2,\ldots, \hat w_{2k}]$. 
In view of Equation (11), it is clear that we have a surjective $H^*(Y;\mathbb Z_2)$-algebra homomorphism $\mathcal R/\mathcal I\to H^*(P(S,X);\mathbb Z_2)$ where  $\hat w_{2j}$ maps to $w_{2j}(\hat \gamma_{n,k})$.  
Therefore it suffices to show that $\mathcal R/\mathcal I$ is isomorphic to $H^*(P(S,X);\mathbb Z_2)$ as an $H^*(Y;\mathbb Z_2)$-module. 

Let $A$ denote the $\mathbb Z_2$-subalgebra of $H^*(Y;\mathbb Z_2)$ generated by $y=w_1(\xi_\alpha)$ and let 
$R=A[\hat w_{2j};1\le j\le k]$.   Then $\mathcal R=H^*(Y;\mathbb Z_2)\otimes_AR.$     Let $I\subset R$ 
denote the ideal generated by $a_{2j}, n-k<j\le n$.  Then $\mathcal R/\mathcal I\cong H^*(Y;\mathbb Z_2)\otimes_A (R/I)$.

Suppose that $(S,\alpha)=(\mathbb S^m,-id)$.  By Proposition \ref{sphere-gnk}, we have $R/I\cong H^*(P(\mathbb S^m;G_{n,k});\mathbb Z_2)\cong A\otimes H^*(G_{n,k};\mathbb Z_2)$.  So $\mathcal R/\mathcal I\cong H^*(Y;\mathbb Z_2)\otimes_A A\otimes H^*(G_{n,k};\mathbb Z_2)
\cong H^*(Y;\mathbb Z_2)\otimes H^*(G_{n,k};\mathbb Z_2)$.  

Next suppose that $(S,\alpha) =(\mathbb S^\infty,-id)$.  The inclusion $\mathbb S^m\hookrightarrow \mathbb S^\infty$ defines an inclusion $j_m: P(\mathbb S^m,X)\hookrightarrow P(\mathbb S^\infty,X)$ which induces a
$H^*(Y;\mathbb Z_2)$-algebra homomorphism in cohomology.  Moreover, $j^*_m$ is an isomorphism up to dimension $m-1$.   From this observation it follows that the theorem holds for $(\mathbb S^\infty, -id)$.  
 In the general case,   the result follows from functoriality and the fact that the $X$-bundle $P(S,X)$ arises as a pull-back of the $X$-bundle $P(\mathbb S^\infty,X)$ via a classifying map $S\to \mathbb RP^\infty$.   This completes the proof.
\end{proof}

Proposition \ref{sphere-gnk} and 
Theorem \ref{gen-grassmann} are valid when the Grassmann manifold is replaced by complex flag manifolds.   More precisely, 
let $\nu :=n_1,\ldots, n_r$ be  an increasing  sequence of positive numbers and let $n=\sum_{1\le j\le r} n_j$. Denote by $F_\nu$ 
the complex flag manifold whose elements are complex vector subspaces $\underline U:=(U_1,\ldots, U_r)$  of $\mathbb C^n$ where $U_i\perp U_j$ for $i\ne j$ and $\dim U_j=n_j, 1\le j\le r$.  Then $F_\nu\cong U(n)/(U(n_1)\times \cdots \times U(n_r))$ has a natural structure of a complex manifold given by the usual inclusion of $U(n)\subset GL_n(\mathbb C)$ so that $F_\nu\cong GL_n(\mathbb C)/P_\nu$ where 
$P_\nu$ is the subgroup which is block upper-triangular, where the diagonal sizes are $n_1,\ldots, n_r$.  (Under this identification, 
$\underline U$ corresponds to the sequence $U_1\subset U_1+U_2\subset \cdots\subset U_1+\cdots+U_r=\mathbb C^n$.) 
 It is well-known that $F_\nu$ has a CW-structure given by {\it Schubert cells} which are all even dimensional.  The Schubert cells are obtained as $B$-orbits of $T$-fixed points 
where $T\subset GL_n(\mathbb C)$ is the diagonal subgroup and $B\subset P_\nu$ is the group of upper triangular matrices.  Their  closures are the {\it Schubert varieties} in $F_\nu$.   
The complex conjugation in $\mathbb C^n$ induces a complex conjugation $\sigma$ on $F_\nu$, and, moreover, $\sigma $ 
stabilizes each Schubert variety.  The fixed points of $\sigma$ is the real flag manifold $\mathbb RF_\nu\cong O(n)/(O(n_1)\times 
\cdots\times O(n_r))$ consisting of $\underline U$ where $U_j\cap \mathbb R^n$ 
is $n_j$-dimensional for all $j$.    
 
We denote 
by $\gamma_{\nu , j}$ (or more briefly $\gamma_j$), the complex vector bundle over $F_\nu$ of rank $n_j$ whose fibre over 
$\underline U$ is the vector space $U_j$.  This is the pull-back of the bundle $\gamma_{n,n_j}$ on $G_{n,n_j}$ via the projection 
$F_\nu\to G_{n,n_j}$ that sends $\underline U$ to $U_j$.   The complex conjugation  $\mathbb C^n$ leads to a $\sigma$-conjugation 
on $\hat \sigma_j$ of $\gamma_j$.   Also  
one has a natural isomorphism of vector bundles
 \[ \bigoplus_{1\le j\le r} \gamma_j\cong n\epsilon_\mathbb C\eqno(12)\]
which respects $\sigma$-conjugation, as in the case of Grassmann manifolds.    The integral cohomology ring of $F_\nu$ is generated 
by $c_{i,j}:=c_i(\gamma_j), 1\le i\le n_j, 1\le j\le r,$ where the only relations among the $c_{i,j}$ are generated by the following 
(inhomogeneous) relation: 
$\prod_{1\le j\le r} c(\gamma_j)=1.$   It follows that $w_{2i,j}=w_{2i}(\gamma_j), 1\le i\le n_j, 1\le j\le r$, generate $H^*(F_\nu;\mathbb Z_2)$ and the 
relations among these generators are all consequences of $\prod_{1\le j\le n_r} w(\gamma_j,t)=1$.  

Suppose that $(S,\alpha)$ is a paracompact space where $\alpha$ is a fixed point free involution.  Then we have the real vector bundles 
$\hat\gamma_j=P(S,\gamma_j)$ over 
the generalized Dold space $P(S,F_\nu)$.  One has the following isomorphism of real vector bundles, resulting from the isomorphism (12):
\[\bigoplus_{1\le j\le r} \hat \gamma_j\cong n\xi_\alpha\oplus n\epsilon_\mathbb R\]
where $\xi_\alpha$ is the real line bundle associated to the double cover $S\times F_\nu\to P(S, F_\nu)$.   Therefore we obtain 
\[ w(\hat \gamma_r,t)=(1+yt)^n.\prod_{1\le j<r} w(\hat \gamma_j,t)^{-1}.\eqno(13)\]  
Then $H^*(P(S,F_\nu);\mathbb Z_2)$ is isomorphic to $H^*(Y;\mathbb Z_2)\otimes H^*(F_\nu;\mathbb Z_2)$ as an $H^*(Y;\mathbb Z_2)$-module, by Proposition \ref{leray-hirsch}.   Arguing as in the proof of Theorem \ref{gen-grassmann}, we observe that 
$H^*(P(S;F_\nu);\mathbb Z_2)$ is generated by as $H^*(Y;\mathbb Z_2)$-algebra by $\hat w_{2i,j}:=w_{2i}(\hat \gamma_j), 1\le i\le n_j, 1\le j<r$.  Using 
Equation (13), we obtain a regular sequence $a_{2s}, n_r<s\le n,$ in the polynomial algebra $R_\nu=\mathbb Z_2[\hat w_{2i,j}\mid 1\le i\le n_j, 1\le j\le r]$ that correspond to the coefficient of $t^{2s} $ in $(1+ty)^n(\prod \hat w_j(t))$ (rewriting $\hat w_{2i+1}$ in terms 
of $y, \hat w_{2l,j}, l\le i,$ using Equation (4)).  

We set $\mathcal R_\nu$ to the polynomial algebra over $H^*(Y;\mathbb Z_2)$ generated by `indeterminates' $\hat w_{2i,j}, 1\le i\le n_j,1\le j< r$, and let 
$\mathcal I_\nu$ be the ideal of $\mathcal R_\nu$ generated by the elements $a_{2i}=a_{2i} (y^2, \hat w_{2l,j}),  n_r<i\le n$.  Then $a_{2i}, n_r<i\le n$, 
form a regular sequence in $\mathcal R_\nu$.  The proof of the following theorem is analogous to that of Theorem \ref{gen-grassmann}.  

\begin{theorem} \label{gen-flagmanifolds} 
Suppose that $(S,\alpha)$ is a paracompact space with a fixed point free involution $\alpha$ and let $\nu=n_1, n_2, \ldots,n_r$ be a sequence 
of positive numbers.    With notations as above, we have an 
isomorphism $\mathcal R_\nu/\mathcal I_\nu\to H^*(P(S,F_\nu);\mathbb Z_2)$ of $H^*(Y;\mathbb Z_2)$-algebras defined by $\hat w_{2i,j}\mapsto w_{2i}(\hat \gamma_j)$. \hfill $\Box$
\end{theorem}

We remark that, setting $k=n-n_r$,  the projection $\pi :F_\nu\to G_{n, k}$  defined as $\underline U\mapsto U_1+\cdots+U_{r-1}$ is $\mathbb Z_2$-equivariant and pulls back $\gamma_{n,k} $ (resp. $\beta_{n,k}$) to $\oplus_{1\le j<r} \gamma_j$ (resp.  $\gamma_r$).   
Hence $P(S,\pi)^*(\hat \gamma_{n,k})=\oplus_{1\le j<r} \hat \gamma_j$.
Moreover, 
using Proposition \ref{leray-hirsch} and the fact that $\pi^*:H^*(G_{n,k};\mathbb Z_2) \to H^*(F_\nu;\mathbb Z_2)$ is a monomorphism, 
we see that $P(S, \pi): P(S,F_\nu)\to P(S,G_{n,k})$ induces a monomorphism in $\mathbb Z_2$-cohomology.    

We conclude this section with the following remark.

\begin{remark}{\em Let $\sigma: S \to S$ be a fixed point free involution.  
Let $X\hookrightarrow F_\nu$ be a Schubert variety in a complex flag manifold $F_\nu$.     As had already been commented, $X$ is stable by complex conjugation $\sigma$ on $F_\nu$ and $X^\sigma=X\cap \mathbb RF_\nu$ is non-empty.   
Moreover, $X$ admits a cell-decomposition having cells only in even dimensions, where the cells are the Schubert cells of $F_\nu$ contained in $X$.   Hence the inclusion $X\hookrightarrow F_\nu$ induces 
a surjection $H^*(F_\nu;\mathbb Z)\to H^*(X;\mathbb Z)$.  It follows that the (mod $2$) cohomology of 
$X$ is generated by Chern classes (mod $2$) of complex vector bundles on $X$.  So Proposition \ref{leray-hirsch} is applicable to the $X$-bundle $P(S,X)\to Y$ and we have $H^*(P(S,X);\mathbb Z_2)\cong H^*(Y;\mathbb Z_2)\otimes H^*(X)$ as $H^*(Y;\mathbb Z_2)$-modules.   As for the $H^*(Y;\mathbb Z_2)$-algebra structure, it is determined by the surjection $H^*(P(S,F_\nu);\mathbb Z_2)\to H^*(P(S,X);\mathbb Z_2)$.   We omit the details.  }
\end{remark}

\subsection{Equivariant cohomology of $(X,\sigma)$} 
As an application of Theorems \ref{torus},  \ref{gen-flagmanifolds} we obtain the $\mathbb Z_2$-equivariant cohomology $H_{\mathbb Z_2}^*(X;\mathbb Z_2)$ of $(X,\sigma)$ when $X$ is either a torus manifold whose torus quotient is a homology polytope, or, is a complex flag manifold $F_\nu$.  At least in the case of complex flag manifolds, this result is perhaps 
known to experts but we could not find an explicit reference.     

When $S=\mathbb S^\infty$ with antipodal action, the space $P(\mathbb S^\infty,X)$ is identical to the Borel construction $\mathbb S^\infty\times_{\mathbb Z_2} X$ (since $\mathbb S^\infty$ is contractible).   
Therefore the equivariant cohomology algbera $H_{\mathbb Z_2}^*(X;\mathbb Z_2)$ equals $H^*(P(\mathbb S^\infty, X);\mathbb Z_2)$.   When $H^*(X;\mathbb Z_2)$ is generated by mod $2$ reduction of Chern classes of finitely many 
$\sigma$-conjugate vector bundles $(\omega_j,\sigma_j)$,  Proposition \ref{leray-hirsch} is applicable and we obtain that 
$H^*_{\mathbb Z_2}(X;\mathbb Z_2)$ is isomorphic to $\mathbb Z_2[y]\otimes 
H^*(X;\mathbb Z_2)$ as a $H^*(\mathbb RP^\infty;\mathbb Z_2)=\mathbb Z_2[y]$-module.  
The inclusion $\mathbb S^m\hookrightarrow \mathbb S^\infty $ induces an inclusion $P(\mathbb S^m,X)\hookrightarrow P(\mathbb S^\infty, X)$ which is an $(m-1)$-equivalence.  
 It follows that $H^i_{\mathbb Z_2}(X;\mathbb Z_2)\cong H^i(P(\mathbb S^\infty,X);\mathbb Z_2)\to H^i(P(\mathbb S^m,X) ;\mathbb Z_2)$ induced by inclusion is an  isomorphism  for all $i<m$. 
Therefore $H^*_{\mathbb Z_2}(X;\mathbb Z_2) $ is isomorphic to the inverse limit of graded $\mathbb{Z}_2$ algebras $\{H^*(P(\mathbb{S}^m,X);\mathbb{Z}_2)\}_{m\ge 2}$.    As an illustration, we obtain the following result 
as an immediate consequence of Theorems \ref{torus} and \ref{gen-flagmanifolds}.  

\begin{theorem}  \label{equivariant-cohomology}  We keep the above notations.   \\
(i) Let $X=X(Q,\Lambda)$ be a $T$-torus manifold where $Q=X/T$ is a homology polytope.  Then $H_{\mathbb Z_2}^*(X;\mathbb Z_2)$ is isomorphic to the $A$-algebra 
$R(Q,\Lambda)$ where $A=H^*(\mathbb RP^\infty;\mathbb Z_2)\cong \mathbb Z_2[y]$.  \\
(ii) Let $\nu=n_1\le \cdots<n_r, n=\sum n_j$.   
Then $H_ {\mathbb Z_2}^*(F_\nu;\mathbb Z_2)$ is 
isomorphic to $\mathcal R_\nu/\mathcal I_\nu $ where $\mathcal R_\nu=A[ \hat w_{2i,j};1\le i\le n_j,1\le j<r]$ where $A=H^*(\mathbb RP^\infty;\mathbb Z_2)=\mathbb Z_2[y]$ and $\mathcal I_\nu\subset \mathcal R_\nu$ 
is the ideal generated by the $a_{2i}\in \mathcal R_\nu, n_r<i\le n$.  \hfill $\Box$
\end{theorem}

{\bf Acknowledgements:}  The authors thank Goutam Mukherjee and Vikraman Uma for their comments on an 
earlier version of this paper.

%%%%%%%%%%%%%%%%%%%%%%%%%%%%%%%%%%%%%%%%%%%%%%
%%%%%%%%%%%%%%%%%%%%%%%%%%%%%%%%%%%%%%%%%%%%%%

\end{document}